\documentclass[a4paper,12pt,reqno]{amsart}
\usepackage[utf8]{inputenc}
\usepackage{amssymb}
\usepackage{amsmath}
\usepackage{amsthm}
\usepackage{natbib}
\usepackage{hyperref}
\hypersetup{
	colorlinks=true,
	linkcolor=black,
	linkbordercolor={1 1 1},
	citecolor=black
}

\newtheorem{theorem}{Theorem}
\newtheorem{lemma}[theorem]{Lemma}

\theoremstyle{definition}
\newtheorem{remark}[theorem]{Remark}

\makeatletter
\@namedef{subjclassname@2020}{\textup{2020} Mathematics Subject Classification}
\makeatother

\newcommand{\IB}{\mathbb{B}}
\newcommand{\IC}{\mathbb{C}}

\newcommand{\IF}{\mathbb{F}}

\newcommand{\IN}{\mathbb{N}}

\newcommand{\IR}{\mathbb{R}}

\newcommand{\cH}{\mathcal{H}}
\newcommand{\cP}{\mathcal{P}}

\newcommand{\Ent}{\mathrm{Ent}}
\newcommand{\vol}{\mathrm{vol}}
\newcommand{\dd}{\mathrm{d}}

\newcommand{\bpb}{\mathbb{B}_{p,\beta}^n}

\title[Next-order asymptotics for the volume of Schatten balls]{Next-order asymptotics for the volume of \\ Schatten balls}
\author[M.~Sonnleitner]{Mathias Sonnleitner}

\address[]{Institute for Mathematical Stochastics, University of Münster, Orl\'eans-Ring 10, 48149 Münster, Germany and Faculty of Mathematics, University of Bielefeld, Universitätsstrasse 25, 33615 Bielefeld, Germany}
\email{mathias.sonnleitner@uni-muenster.de}


\begin{document}

\begin{abstract}
	The volume of the unit balls of self-adjoint finite-dimensional Schatten $p$-classes of $n\times n$-matrices, $1\le p\le \infty$, is only known exactly for $p=2$ and $p=\infty$. We give an asymptotic expansion of the logarithmic volume to order $o(n)$ for all $p>1$. The proof rests on asymptotics for the partition function of $\beta$-ensembles due to Lebl\'e and Serfaty [Invent. Math. 210(3):645--757, 2017]. Independently, the case $p\ge 2$ was obtained by Dworaczek Guera, Memin and Pain [arXiv:2511.05386]. In the complex case the asymptotic expansion is continued to order $o(1)$ for all $p\ge 1$. 
\end{abstract}

\maketitle

\section{Introduction}

The Schatten or Schatten-von Neumann class consists of all compact operators between two Hilbert spaces whose sequence of singular values belongs to the Lebesgue sequence space $\ell_p$. It can be seen a non-commutative version of $\ell_p$ and is named after Schatten and von Neumann \cite{Sch46,Sch50,SVN46}. We refer to \cite{Pie07,PX03} for an overview of their role in functional analysis and Banach space geometry.

In this article, we study the geometry of finite-dimensional Schatten classes and in particular their unit balls. These have been investigated as test cases for conjectures in asymptotic geometric analysis \cite{DFG+23,GMP25,GP07,KMP98,RV20}, in the context of low-rank matrix recovery and information-based complexity \cite{CR12,CK15,HPV17,HPV21,PS22} or in quantum information theory \cite{ASW11,Wil13}. Let $\mathbb{F}_1=\IR$, $\mathbb{F}_2=\IC$ and $\mathbb{F}_4=\mathbb{H}$ and denote the singular values of a matrix $A\in \mathbb{F}_{\beta}^{n\times n}$, where $\beta\in\{1,2,4\}$, by 
\[
s_1(A)\ge \cdots \ge s_n(A)\ge 0.
\] 
For $1\le p\le \infty$, the $p$-Schatten norm of $A$ is 
\[
\|A\|_p
=
\begin{cases}
	\Big(\sum_{j=1}^{n}s_j(A)^p\Big)^{1/p}&\colon p<\infty\\
	s_1(A) &\colon p=\infty.
\end{cases}
\]

The associated unit ball and its intersection with the subspace of self-adjoint matrices $\cH_n(\IF_{\beta})=\{A\in \IF_{\beta}^{n\times n}\colon A^{*}=A\}$  are denoted by
\[
B_{p,\beta}^n
=\{A\in \IF_{\beta}^{n\times n}\colon \|A\|_p\le 1\}\quad \text{and}\quad
\bpb
= B_{p,\beta}^n \cap \cH_n(\IF_{\beta}).
\]
For self-adjoint matrices the Schatten $p$-norm $\|A\|_p$ becomes a $p$-norm of eigenvalues. In case of non-self adjoint matrices one can also study the rectangular case, see e.g. \cite{JKP24}.

We shall identify $\IF_{\beta}^{n\times n}$ with $\IR^{\beta n^2}$ and $\cH_n(\IF_{\beta})$ with $\IR^{d_n}$, where 
\begin{align}\label{eq:dn}
d_n:=\beta\frac{n(n-1)}{2}+n.
\end{align}
Then $B_{p,\beta}^n$ and $\bpb$ are convex bodies (i.e. compact convex sets with nonempty interior) in $\IR^{\beta n^2}$ and $\IR^{d_n}$, respectively. In the following, we study their volumes $\vol(B_{p,\beta}^n)$ and $\vol(\bpb)$. Our aim is to give more precise asymptotics of the volume of $\bpb$ as $n\to \infty$.

The exact volume of $B_{p,\beta}^n$ and $\bpb$ is only known in the cases $p=2$ and $p=\infty$. For $p=2$, the Schatten $2$-balls are Euclidean unit balls of corresponding dimension and thus
\begin{equation*} 
\vol(B_{2,\beta}^n)
=\frac{\pi^{\beta n^2/2}}{\Gamma(1+\frac{\beta n^2}{2})}
\qquad\text{and}\qquad
\vol(\IB_{2,\beta}^n)
=\frac{\pi^{d_n/2}}{\Gamma(1+\frac{d_n}{2})},
\end{equation*} 
where $\Gamma(x)=\int_0^{\infty}t^{x-1}e^{-t}\dd t$ denotes the Gamma function. For $p=\infty$ it follows from Saint Raymond's work \cite{SR84} and Selberg's integral formula \cite{And91,Sel44} (see also Remark~\ref{rem:infty-non-sa} below) that
\begin{equation} \label{eq:inf-exact}
\vol(B_{\infty,\beta}^n)
=\frac{\prod_{j=0}^{n-1}\Gamma( 1+j\frac{\beta}{2} )}{\prod_{j=n}^{2n-1}\Gamma( 1+j\frac{\beta}{2})}\pi^{\beta n^2/2},
\end{equation}
and in the self-adjoint case it holds that
\begin{equation} \label{eq:inf-sa-exact}
\vol(\IB_{\infty,\beta}^n)
=2^{d_n}(2\pi)^{\beta n(n-1)/4}\prod_{j=0}^{n-1}\frac{\Gamma( 1+j\frac{\beta}{2})^2}{\Gamma( 2+(n+j-1)\frac{\beta}{2} )}.
\end{equation}
For general $1\le p\le \infty$, exact volumes are unknown and instead their asymptotic behavior has been studied. Saint Raymond \cite{SR84} derived $A(p)>0$ such that, for $\beta\in\{1,2\}$, 
\begin{equation} \label{eq:sim-asym}
\lim_{n\to \infty}n^{1/2+1/p}\vol(B_{p,\beta}^n)^{1/\beta n^2}
=\sqrt{\frac{2\pi}{\beta}}e^{3/4}A(p)
\end{equation}
and determined $A(2)=e^{-1/4}$ and $A(\infty)=\frac{1}{2}$. Gu\'{e}don and Paouris \cite{GP07} showed that $\vol(B_{p,4}^n)^{1/4 n^2}$ and $\vol(\bpb)^{1/d_n}, \beta\in\{1,2,4\}$ are of the same order $n^{-1/2-1/p}$. Kabluchko, Prochno and Thäle \cite{KPT20a} determined
\begin{equation} \label{eq:delta}
A(p)
=\frac{1}{2}\bigg(\frac{p\sqrt{\pi}\Gamma(\frac{p}{2})}{\sqrt{e}\Gamma(\frac{p+1}{2})}\bigg)^{1/p}
\end{equation}
and moreover showed in \cite{KPT20b} that, for $\beta\in \{1,2,4\}$,
\begin{equation} \label{eq:sim-sym}
	\lim_{n\to \infty}n^{1/2+1/p}\vol(\bpb)^{1/d_n}
	=\sqrt{\frac{4\pi}{\beta}}e^{3/4}A(p).
\end{equation}
We remark that one way to see that $A(p)$ is decreasing in $p$ is to observe that the same is true for the normalized volume $\vol(n^{1/p}\bpb)$. Recently, Dadoun, Fradelizi, Gu\'{e}don and Zitt \cite[Remark 4.8]{DFG+23} showed \eqref{eq:sim-asym} for $\beta=4$. Thus, the asymptotics in \eqref{eq:sim-asym} and \eqref{eq:sim-sym} hold in all cases $\beta\in \{1,2,4\}$ and $1\le p\le \infty$. 

In case of $p=\infty$, the quantity $A(\infty)=\frac{1}{2}$ can be expressed in terms of the minimal logarithmic energy of probability measures supported in $[-1,1]$ attained by the arcsine distribution, see \cite{ST97}. We refer to \cite{BDF+24} and \cite{Bra24} for related asymptotics for discrete minimal logarithmic energy. 

In case of $p<\infty$, the quantity $A(p)$ in \eqref{eq:delta} can be expressed as 
\[
A(p)
=\frac{1}{2}\Big(\frac{p\sqrt{\pi}\Gamma(\frac{p}{2})}{\sqrt{e}\Gamma(\frac{p+1}{2})}\Big)^{1/p}
=\frac{1}{2}\Big(\frac{1}{\sqrt{e}\alpha_p}\Big)^{1/p},
\]
where $\alpha_p=\int_{\IR} |x|^p\dd\mu_p(x)$ is the $p$-th absolute moment of the Ullman distribution $\mu_p$ which has density
\begin{equation} \label{eq:ullman-1}
f_p(x)
= \frac{p}{\pi}\int_{|x|}^1 \frac{t^{p-1}}{\sqrt{t^2-x^2}}\dd t, \qquad x\in [-1,1].
\end{equation}
The Ullman distribution originated in the theory of orthogonal polynomials on $\IR$ with weight $e^{-c|x|^p}$ and $A(p)$ is strongly related to the asymptotics of the sequence of leading coefficients, see \cite{ST97} and \cite[Sec. 4.1]{Ass87}. 

Our main contribution is the following asymptotic expansion for the logarithmic volume $\ln \vol(\bpb)$. Recall that the differential entropy of a probability measure $\mu$ on $\IR$ is
\[
\Ent(\mu)
= -\int_{\IR}\ln(f(x))f(x)\dd x, 
\]
whenever $\mu$ admits a Lebesgue density $f$ and $\Ent(\mu)=-\infty$ otherwise. 

\begin{theorem} \label{thm:main}
Let $\beta\in \{1,2,4\}$. If $p>1$, then, as $n\to\infty$, 
\begin{align}\label{eq:main-1}
\ln \vol(\bpb)
=&-\frac{\beta}{2}\Big(\frac{1}{2}+\frac{1}{p}\Big)n^2\ln n\notag\\
&+\frac{\beta}{2}\Big(\frac{1}{2}\ln\frac{4\pi}{\beta}+\frac{3}{4}+\ln A(p)\Big)n^2\notag\\
&-\Big(1-\frac{\beta}{2}\Big)\Big(\frac{1}{2}+\frac{1}{p}\Big)n\ln n\notag\\
&+\Big(1-\frac{\beta}{2}\Big)\Big(\frac{1}{2}\ln\frac{4}{\beta\pi}+\frac{1}{2}+\frac{1}{2p}+\ln A(p)+\Ent(\mu_p)\Big)n +o(n).
\end{align}
Moreover, if $\beta=2$ and $p\ge 1$, then
\begin{equation} \label{eq:main-2}
\ln \vol(\IB_{p,2}^n)
=-\Big(\frac{1}{2}+\frac{1}{p}\Big)n^2\ln n
+\Big(\frac{1}{2}\ln 2\pi +\frac{3}{4}+\ln A(p)\Big)n^2
-\ln n+M_{p}+o(1),
\end{equation}
where
\[
M_{p}=\frac{5}{12}\ln p+\frac{1}{12}\ln 2.
\]
\end{theorem}

In case of $p=2$ it is known that $\Ent(\mu_2)=\ln \pi -\frac{1}{2}$ where $\mu_2$ is the semi-circle law. Otherwise the entropy $\Ent(\mu_p)$ of the Ullman distribution appears to be unknown. Independently, the expansion in \eqref{eq:main-1} was recently obtained by Dworaczek Guera, Memin and Pain \cite{GMP25} in case of $p\ge 2$. 

In addition to Theorem~\ref{thm:main} we obtain for $p=\infty$ from the exact formula \eqref{eq:inf-sa-exact} and asymptotics for the Barnes $G$-function satisfying $G(n+1)=\prod_{k=0}^{n-1}\Gamma(1+k)$, $n\in \IN$, that
\begin{align}\label{eq:inf-sa-asymp}
\ln \vol(\IB_{\infty,2}^n)
=&-\frac{1}{2}n^2\ln n 
+\Big(\frac{1}{2}\ln 2\pi +\frac{3}{4}+\ln A(\infty)\Big)n^2\notag\\
&-\frac{1}{6}\ln n +\frac{1}{12}\ln 2+2\zeta'(-1)+o(1),
\end{align}
where $\zeta$ is the Riemann zeta function and $\zeta'(-1)=\frac{1}{12}-\ln A$ with $A=1.282...$ being the Glaisher-Kinkelin constant. The logarithmic term is what we would expect from 
\[
\ln \vol(\IB_{p,\beta}^n)
\le \ln \vol(\IB_{\infty,\beta}^n)
\le \ln \vol(\IB_{p,\beta}^n)+\frac{d_n}{p}\ln n,
\]
upon setting $p=n^2(\ln n)^{2}$ such that the difference between the logarithmic volumes vanishes. The discrepancy in the constant terms in \eqref{eq:main-2} and \eqref{eq:inf-sa-asymp} is because the $o(1)$-term in \eqref{eq:main-2} implicitly depends on $p$.

The proof of Theorem~\ref{thm:main} relies on the fact that the volume of $\bpb$ can be expressed in terms of the integral
\begin{equation} \label{eq:Z}
Z_{n,p,\beta}
=\int_{\IR^n}\prod_{1\le i<j\le n}|x_i-x_j|^{\beta} \prod_{i=1}^n e^{-\frac{\beta}{2}n\, v_p|x_i|^p}\dd x,
\end{equation}
where
\[
v_p =\frac{\sqrt{\pi}\Gamma(\frac{p}{2})}{\Gamma(\frac{p+1}{2})}=\frac{1}{p\alpha_p}
\]
is chosen for normalization purposes. The quantity $Z_{n,p,\beta}$ is the partition function of a one-dimensional $\beta$-ensemble with potential $V(x)=v_p |x|^p$, where eigenvalues/particles are distributed with joint density 
\begin{equation} \label{eq:density}
f_{n,p,\beta}(x)
=\frac{1}{Z_{n,p,\beta}}\prod_{1\le i<j\le n}|x_i-x_j|^{\beta} \prod_{i=1}^n e^{-\frac{\beta}{2}n\, v_p|x_i|^p}, \qquad x\in \IR^n.
\end{equation}
The parameter $\beta$ stands for the inverse temperature and may take any positive value, while for the matrix representation we restrict to $\beta\in \{1,2,4\}$. In case of $p=2$ the integral $Z_{n,2,\beta}$ is known exactly via the Mehta-Selberg integral. We refer to \cite{AGZ10} and \cite{For10} for more information on $\beta$-ensembles. 

In this context the Ullman distribution arises as the corresponding equilibrium measure minimizing the functional
\[
\mu\mapsto I_p(\mu)
=-\int_{\IR}\int_{\IR} \ln|x-y|\dd\mu(x)\dd\mu(y)+v_p\int_{\IR} |x|^p\dd\mu(x)
\]
over $\mu\in \cP(\IR)$, where $\cP(\IR)$ denotes the space of probability measures on $\IR$.  More precisely, if $X=(X_1,\dots,X_n)$ is distributed according to the density $f_{n,p,\beta}$, then the random empirical measure $ \mu_{n,p} =\frac{1}{n}\sum_{i=1}^{n}\delta_{X_i} $  tends to $\mu_p$ as $n\to\infty$. Moreover, large deviation principles due to \cite{BG97} and \cite{HP00} yield 
\begin{equation} \label{eq:Z-first}
\lim_{n\to\infty}\frac{2}{\beta n^2}\ln Z_{n,p,\beta}
= I_p(\mu_p)
=\ln 2+\frac{3}{2p}.
\end{equation}
Earlier work \cite{DPS95,Joh98} also identified the limiting constant but with different techniques. In \cite{DFG+23}, the limit \eqref{eq:Z-first} is used to prove \eqref{eq:sim-sym}. The large deviation principle for the corresponding tagged empirical fields due to Lebl\'e and Serfaty \cite{LS17} allows to refine \eqref{eq:Z-first} as follows. 

\begin{theorem} \label{thm:LS}
Let $\beta>0$ and $p>1$. As $n\to\infty$,
\[
\ln Z_{n,p,\beta}
=-\frac{\beta}{2}n^2 I_p(\mu_p)+\frac{\beta}{2}n\ln n- C(\beta)n+\Big(1-\frac{\beta}{2}\Big)\Ent(\mu_p)n +o(n),
\]
where 
\begin{equation} \label{eq:c-beta}
C(\beta)
=\Big(1-\frac{\beta}{2}\Big)\ln \frac{\beta}{2}-\frac{\beta}{2}\ln 2\pi+\frac{\beta}{2}+\ln \Gamma\Big(\frac{\beta}{2}\Big).
\end{equation}
\end{theorem}

Theorem~\ref{thm:LS} forms the core of the proof of \eqref{eq:main-1} in Theorem~\ref{thm:main}. Independently, in the recent work \cite{GMP25} the authors deduced the statement of Theorem~\ref{thm:LS} in case of $p\ge 2$ from a central limit theorem for linear eigenvalue statistics. The restriction to $p\ge 2$ results from regularity requirements of the underlying approach, see also \cite{DFG+23}. The milder restriction to $p>1$ in Theorem~\ref{thm:LS} is due to the fact that in this case the Ullman density $f_p$ is Hölder continuous inside its support which is already sufficient for the approach in \cite{LS17}.

The term $\Ent(\mu_p)$ in Theorem~\ref{thm:LS} arises via a rescaling property of the specific entropy which is one part of the rate function in the large deviation principle in \cite{LS17}. In case of $\beta=2$, the contribution of $\Ent(\mu_p)$ vanishes and more precise asymptotics hold. The following is a consequence of \cite[Prop.~1.1]{CKM23}, derived from \cite{KM99}, and \cite[Thm.~1.9]{CKM23}.

\begin{theorem} \label{thm:CKM}
Let $p\ge 1$. As $n\to\infty$,
\[
\ln Z_{n,p,2}
=-n^2 I_p(\mu_p)+n\ln n+(\ln 2\pi -1) n +\frac{5}{12}\ln n -\frac{1}{12}\ln \frac{p}{2}+\zeta'(-1)+o(1),
\]
where $\zeta$ is the Riemann zeta function. 
\end{theorem}

Theorem~\ref{thm:CKM} combined with Lemma~\ref{lem:vol-partition} and Lemma~\ref{lem:cn} below yields \eqref{eq:main-2} in Theorem~\ref{thm:main}. In case of $p$ an even integer, the potential $V(x)=v_p |x|^p$ is an analytic function and \cite[Prop.~1.2]{BG13} yields a full asymptotic expansion of $\ln Z_{n,p,\beta}$, and consequently of $\ln \vol(\bpb)$, with coefficients defined in terms of integrals of the Stieltjes transform of $\mu_p$, see also \cite[Sec.~7]{BG24} and \cite{Shc13,Shc14}. 

In the case of non self-adjoint matrices, for $\beta\in\{1,2,4\}$ and $1\le p<\infty$, the volume of the unit ball $B_{p,\beta}^n$ may be written in terms of 
\[
\tilde{Z}_{n,p,\beta}
=\int_{(\IR_+)^n}\prod_{1\le i<j\le n}|x_i-x_j|^{\beta}\prod_{i=1}^n x_i^{\beta/2-1}\prod_{i=1}^n e^{-\beta n\, v_px_i^{p/2}}\dd x,
\]
see e.g.~\cite[Sec.~4.2]{DFG+23}. This corresponds to the partition function of a Laguerre $\beta$-ensemble with potential $V(x)=v_p|x|^{p/2}$. The leading order asymptotics for $\tilde{Z}_{n,p,\beta}$ can again be obtained via large deviations, leading to \eqref{eq:sim-asym}. However, the results in \cite{LS17} do not apply because of a blow-up of the corresponding equilibrium density near the boundary of its support.  It goes beyond the scope of the present work to derive next-order asymptotics for $\tilde{Z}_{n,p,\beta}$ and $B_{p,\beta}^n$. For completeness, we note that in case of $p=\infty$ it can be derived that (see e.g. \cite{JKP24})
\begin{align*}
\ln \vol(B_{\infty,\beta}^n)
=&-\frac{\beta}{2}n^2\ln n +\beta\Big(\frac{1}{2}\ln\frac{2\pi}{\beta}+\frac{3}{4}+\ln A(\infty)\Big)n^2\\
&+\Big(1-\frac{\beta}{2}\Big)\ln(A(\infty))n
+\frac{-3+\frac{\beta}{2}+\frac{2}{\beta}}{12}\ln n
+O(1).
\end{align*}

The unit balls $B_{p,\beta}^n$ and $\bpb$ can even be studied for $0<p<1$. Then \eqref{eq:sim-asym} and \eqref{eq:sim-sym} as well as Theorem~\ref{thm:CKM} (with a possibly modified constant term) continue to hold. However, the proof of Theorem~\ref{thm:main} and the underlying large deviation principle in \cite{LS17} appear to be limited to $p\ge 1$. This is because the rate function in \cite{LS17} is the free energy which is uniquely minimized by the sine$_{\beta}$-process, see \cite{EHL21,HM25}. While the microscopic limit point process in the bulk of many $\beta$-ensembles is given by the universal sine$_{\beta}$-process, see \cite{BEY14,Meh91,VV09}, for $\beta=2$ and $p<1$ the density $f_p$ develops a singularity at zero near which the limit point process depends on $p$, see \cite{CKM23}. In case of $\beta\in \{1,4\}$, it is an open problem to extend Theorem~\ref{thm:main} to $p=1$ for which
\[
f_1(x)
=\frac{1}{\pi}\ln \frac{1+\sqrt{1-x^2}}{|x|},\qquad x\in [-1,1].
\]
The logarithmic singularity at the origin prevents direct application of the results in \cite{LS17} but possibly behaves well enough such that the approach may be adapted to this case. 

\begin{remark}
	Heuristically, the unit balls $B_{p,\beta}^n$ and $\IB_{p,\beta}^n$ are non-commutative versions of $\IB_p^n=\{x\in \IR^n\colon \sum_{i=1}^{n}|x_i|^p\le 1\}$. To compare, note that
\[
\vol(n^{1/p}\IB_{p}^n)
=n^{n/p}\frac{2^n\Gamma(1+\frac{1}{p})^n}{\Gamma(1+\frac{n}{p})}
=\frac{1+o(1)}{\sqrt{2\pi n/p}} e^{\Ent(\nu_p)n},
\]
where $\nu_p$ has density proportional to $e^{-|x|^p/p}$ and maximizes entropy over $\mu\in \cP(\IR)$ with $\int |x|^p\dd\mu(x) \le 1$, see also \cite{KP21}. Similarly, the Ullman distribution $\mu_p$ maximizes free entropy over $\mu\in \cP(\IR)$ with $\int |x|^p\dd\mu(x)\le \alpha_p$, see \cite[Prop.~5.34]{HP00}. In case of $p=\infty$, one can interpret $\nu_{\infty}$ as the uniform distribution on $[-1,1]$, which maximizes entropy, and $\mu_{\infty}$ as the arcsine distribution, which maximizes free entropy. Formally, see Lemma~\ref{lem:vol-partition} below, the case $\beta=0$ corresponds to the commutative case where a Poisson point process appears in the microscopic limit, see also \cite{AD14,BP15,Leb16}. Possibly, a careful analysis of the behavior of the asymptotics for $\ln Z_{n,p,\beta}$ as $\beta\to 0$ could allow a computation of $\Ent(\mu_p)$ in terms of $\Ent(\nu_p)$.
\end{remark}

\begin{remark}
Theorem~\ref{thm:main} can be combined with the stochastic representation in \cite[Cor.~4.3]{KPT20b} of a random matrix drawn uniformly from a Schatten ball $\bpb$ to obtain information on critical intersections of Schatten $p$-balls in the spirit of \cite{KPT20b} and \cite{ST25}. 
\end{remark}

\section{Proofs}

In the following, we provide the proofs of Theorems~\ref{thm:main} and \ref{thm:LS}. First, in Lemma~\ref{lem:hoelder} we prove Hölder regularity of the Ullman density. Then Theorem~\ref{thm:LS} is a direct consequence of \cite{LS17}. Subsequently, in Lemma~\ref{lem:vol-partition} we recall a Weyl type integration formula providing the link between $\vol(\bpb)$ and $Z_{n,p,\beta}$. Finally, we deduce Theorem~\ref{thm:main} from Theorem~\ref{thm:LS} with the help of a known asymptotic expansion in Lemma~\ref{lem:cn}. 

Corollary 1.1 in \cite{LS17} states that the asymptotics in Theorem~\ref{thm:LS} hold, i.e.
\begin{equation} \label{eq:Z-LS}
\ln Z_{n,p,\beta}
=-\frac{\beta}{2}n^2 I_p(\mu_p)+\frac{\beta}{2}n\ln n-  C(\beta)n+\Big(1-\frac{\beta}{2}\Big)\Ent(\mu_p)n +o(n),
\end{equation}
provided that the function $V(x)=v_p |x|^p$ satisfies certain assumptions (H1)-(H5). These involve regularity of the density of the equilibrium measure which in this case is the Ullman density
\[
f_p(x)
= \frac{p}{\pi}\int_{|x|}^1 \frac{t^{p-1}}{\sqrt{t^2-x^2}}\dd t,\qquad x\in [-1,1].
\]
It is known from \cite[Thm.~1.34]{DKM98} that if $V\in C^3$, that is if $p\ge 3$, then the density of the equilibrium measure is Hölder continuous of order $\frac{1}{2}$ and behaves like $\sqrt{1-|x|}$ near the boundary of its support. To include smaller values of $p$, we verify the following statement directly. 

\begin{lemma} \label{lem:hoelder}
	Let $p>1$. Then $f_p$ is Hölder continuous of order $\min\{p-1,\frac{1}{2}\}$ on $[-1,1]$. Moreover, it is Hölder continuous of order $\frac{1}{2}$ in neighborhoods of $-1$ and $1$ and there exist positive constants $c$ and $C$ such that 
\begin{equation} \label{eq:tail}
	c\sqrt{1-|x|}\le f_p(x) \le C\sqrt{1-|x|}\qquad \text{for}\qquad |x|\ge\frac{1}{2}.
\end{equation}
\end{lemma}
\begin{proof}
First, we verify the behavior of $f_p$ near the boundary points. By a change of variables, for $x\in (-1,1)\setminus\{0\}$, 
\begin{equation} \label{eq:change}
\frac{\pi}{p}f_p(x)
=\int_{|x|}^1 \frac{t^{p-1}}{\sqrt{t^2-x^2}}\dd t
=|x|^{p-1}\int_0^{\arccos |x|}\frac{1}{\cos^p \alpha}\dd \alpha.
\end{equation}
Combining \eqref{eq:change} with the asymptotics $\arccos |x|=\sqrt{2(1-|x|)}+o(1-|x|)$ as $|x|\to 1$ proves \eqref{eq:tail} which in turn implies that $f_p$ is Hölder continuous of order $\frac{1}{2}$ at the boundary points $\pm 1$.

Next, we show that $f_p$ is in fact continuously differentiable at $x\in (-1,1)\setminus \{0\}$. For symmetry reasons it is enough to consider $x\in (0,1)$. Let $0<|h|<\min\{x,1-x\}$ and write 
\begin{align*}
\frac{\pi}{p}\big(f_p(x+h)-f_p(x)\big)	
=&\big((x+h)^{p-1}-x^{p-1}\big)\int_0^{\arccos x}\frac{1}{\cos^p\alpha}\dd\alpha\\
&-(x+h)^{p-1}\int_{\arccos(x+h)}^{\arccos x}\frac{1}{\cos^p \alpha}\dd \alpha.
\end{align*}
Using $\arccos(x+h)=\arccos x -\frac{h}{\sqrt{1-x^2}}+o(h)$ and letting $h\to 0$ shows that $f_p$ is continuously differentiable at $x$ with derivative
\[
f_p'(x)
=\frac{p-1}{x}f_p(x)-\frac{p}{\pi}\frac{1}{x\sqrt{1-x^2}}.
\]

It remains to check Hölder continuity at zero. Let $h\in (0,1)$ and write
\begin{align*}
\frac{\pi}{p}\big(f_p(h)-f_p(0)\big)
&=\int_h^1 \frac{t^{p-1}}{\sqrt{t^2-h^2}}\dd t - \int_0^1 t^{p-2}\dd t\\
&=\int_h^1 t^{p-1}\Big(\frac{1}{\sqrt{t^2-h^2}}-\frac{1}{t}\Big)\dd t -\frac{h^{p-1}}{p-1}.
\end{align*}
The latter integral can be estimated by
\begin{align*}
\int_h^1 t^{p-1}\Big(\frac{1}{\sqrt{t^2-h^2}}-\frac{1}{t}\Big)\dd t
&\le h^2\int_h^1 \frac{t^{p-3}}{\sqrt{t^2-h^2}}\dd t\\
&= h^{p-1}\int_1^{1/h} \frac{u^{p-3}}{\sqrt{u^2-1}}\dd u.
\end{align*}
In particular, if $1< p<2$, then
\begin{equation} \label{eq:hoelder}
0\le \frac{\pi}{p}\frac{f_p(h)-f_p(0)}{h^{p-1}}
\le \int_1^{\infty} \frac{u^{p-3}}{\sqrt{u^2-1}}\dd u
<\infty.
\end{equation}
Thus, in this case we obtain that $f_p$ is Hölder continuous of order $p-1$ at $0$. For $p\ge 2$, a similar argument shows that $f_p$ is continuously differentiable at $0$ with $f_p'(0)=0$. In conclusion, if $p>1$, the function $f_p$ is Hölder continuous of order $\min\{p-1,\frac{1}{2}\}$ on $[-1,1]$. This completes the proof.
\end{proof}

We can now deduce Theorem~\ref{thm:LS}.

\begin{proof}[Proof of Theorem~\ref{thm:LS}]
	Let $p>1$. In order to apply \cite[Cor.~1.1]{LS17} to $V(x)=v_p|x|^p$, we verify assumptions (H1)-(H5) there. The function $V$ is lower semi-continuous, bounded from below, everywhere finite and grows polynomially as $|x|\to \infty$. Thus, assumptions (H1) and (H2) in \cite{LS17} are satisfied. By Lemma~\ref{lem:hoelder} the density $f_p$ of the equilibrium measure $\mu_p$ is Hölder continuous of order $\min\{p-1,\frac{1}{2}\}$ on its support $\Sigma=[-1,1]$ (and Hölder continuous of order $\frac{1}{2}$ away from zero). Thus, assumption (H3) is satisfied. For assumption (H4) we note that $\Sigma$ is connected and its boundary is given by $\{-1,1\}$. Further, by Lemma~\ref{lem:hoelder}, the density $f_p$ behaves like $c\sqrt{1-|x|}$ in neighborhoods of the boundary points and is Hölder continuous of order $\frac{1}{2}$ there. Thus, assumption (H4) is satisfied. Assumption (H5) is satisfied since $x\mapsto \exp(-c(|x|^p-\ln |x|))$ is integrable on $\IR$ for every $c>0$. Thus, we can apply \cite[Cor.~1.1]{LS17} to obtain \eqref{eq:Z-LS}.  Finally, the constant $C(\beta)$ can be obtained from the known full expansion of the partition function in the Gaussian case $p=2$, see e.g. \cite[Sec.~7]{BG24}. 
\end{proof}

For the proof of Theorem~\ref{thm:main} we provide the following well-known link between Schatten $p$-balls and $\beta$-ensembles. It is a combination of a Weyl type integration formula, see \cite{AGZ10,SR84}, with a classical trick in convexity going back to \cite{Dir39} for computing the volume of the unit ball of $\ell_p^n$, see \cite{DFG+23,GP07,KMP98}. Recall that $d_n=\frac{\beta}{2}n^2+(1-\frac{\beta}{2})n$ and $\IB_p^n=\{x\in \IR^n\colon \sum_{i=1}^{n}|x_i|^p\le 1\}$.

\begin{lemma} \label{lem:vol-partition}
Let $\beta\in\{1,2,4\}$ and $1\le p<\infty$. Then
\begin{align*}
\vol(\bpb)
&=c_n\int_{\IB_p^n} \prod_{1\le i< j\le n}|x_i-x_j|^{\beta}\dd x\\
&=\frac{c_{n}}{\Gamma(1+\frac{d_n}{p})}\Big(\frac{n\beta v_p}{2}\Big)^{d_n/p}Z_{n,p,\beta}
\end{align*}
where $Z_{n,p,\beta}$ and $v_p$ are as in \eqref{eq:Z} and
\[
c_{n}
=\frac{1}{n!}\bigg(\frac{\Gamma(\frac{\beta}{2})}{(2\pi)^{\beta/2}}\bigg)^{n}\prod_{k=1}^n\frac{(2\pi)^{\beta k/2}}{\Gamma(\frac{\beta k}{2})}.
\]
\end{lemma}

\begin{remark}\label{rem:infty-non-sa}
For $p=\infty$, the statement of Lemma~\ref{lem:vol-partition} becomes
\[
\vol(\IB_{\infty,\beta}^n)
=c_n\int_{[-1,1]^n} \prod_{1\le i< j\le n}|x_i-x_j|^{\beta}\dd x.
\]
Then \eqref{eq:inf-sa-exact} is a consequence of Selberg's integral formula.
\end{remark}

As a final ingredient, we shall need the following well-known asymptotics for $c_{n}$ as in Lemma~\ref{lem:vol-partition}. 

\begin{lemma} \label{lem:cn}
Let $\beta\in \{1,2,4\}$. Then, as $n\to\infty$,
\begin{align*}
\ln c_{n}
=&-\frac{\beta}{4}n^2\ln n + \frac{\beta}{2}\Big(\frac{1}{2}\ln \frac{4\pi}{\beta}+\frac{3}{4}\Big)n^2-\frac{1}{2}\Big(1+\frac{\beta}{2}\Big)n\ln n\\
&+\Big(\frac{1}{2}\Big(1-\frac{\beta}{2}\Big)\ln \beta -\frac{1}{2}\Big(1+\frac{\beta}{2}\Big)(\ln \pi-1)-\ln 2+\ln \Gamma\Big(\frac{\beta}{2}\Big)\Big)n\\
&-\frac{3+\frac{\beta}{2}+\frac{2}{\beta}}{12}\ln n
+a_{\beta}+o(1),
\end{align*}
where $a_{\beta}\in \IR$ is some constant with $a_2=\frac{1}{2}\ln 2\pi-\zeta'(-1)$ with $\zeta$ being the Riemann zeta function.
\end{lemma}
\begin{proof}
It follows from the definition of $c_{n}$ that 
\[
c_{n}
=\frac{1}{n!}\Gamma_2\Big(n+1;\frac{2}{\beta},1\Big)\bigg(\frac{\Gamma(\frac{\beta}{2})}{(\frac{\beta}{2})^{\frac{\beta}{2}}}\bigg)^n\Big(\frac{4\pi}{\beta}\Big)^{\frac{\beta}{4}n^2-(\frac{1}{2}+\frac{\beta}{4})n},
\]
where $\Gamma_2(x;\frac{2}{\beta},1)$ is a particular instance of the Barnes double Gamma function, see \cite{BG24,Spr09}. Adapting \cite[eq. (7.17)]{BG24}, the following asymptotic expansion holds
\begin{align*}
\ln \Gamma_2\Big(n+1;\frac{2}{\beta},1\Big)	
=&-\frac{\beta}{4}n^2\ln n+\frac{3\beta}{8}n^2+\frac{1}{2}\Big(1-\frac{\beta}{2}\Big)n\ln n-\frac{1}{2}\Big(1-\frac{\beta}{2}\Big)n\\
&+\bigg(\frac{1}{2} -\frac{3+\frac{\beta}{2}+\frac{2}{\beta}}{12}\bigg)\ln n +a_{\beta}'+o(1),
\end{align*}
where $a_{\beta}'\in \IR$ is some constant with $a_{2}'=\frac{1}{2}\ln 2\pi-\zeta'(-1)$. Combined with Stirling's approximation $\ln n!=n\ln n-n+\frac{1}{2}\ln(2\pi n)+o(1)$ we deduce the claimed expansion.
\end{proof}

We are now ready to deduce the asymptotics for the volume of the self-adjoint Schatten classes $\bpb$ given in Theorem~\ref{thm:main}.

\begin{proof}[Proof of Theorem~\ref{thm:main}]
Let $\beta\in \{1,2,4\}$ and $p>1$. Recall that $d_n=\frac{\beta}{2}n^2+(1-\frac{\beta}{2})n$. By Lemma~\ref{lem:vol-partition} it holds that
\begin{equation} \label{eq:vol-asymp-parts}
\ln \vol(\bpb)
=\ln c_{n}+\frac{d_n}{p}\ln \Big(\frac{\beta nv_p}{2}\Big)-\ln \Gamma\Big(1+\frac{d_n}{p}\Big)+\ln Z_{n,p,\beta}.
\end{equation}
By Stirling's approximation, as $n\to\infty$,
\begin{align}\label{eq:p-term-1}
\ln \Gamma\Big(1+\frac{d_n}{p}\Big)
&=\frac{d_n}{p}\Big(\ln \frac{d_n}{p}-1\Big)+\frac{1}{2}\ln\Big(\frac{2\pi d_n}{p}\Big)+o(1)\notag\\
&=\frac{d_n}{p}\Big(2\ln n+\ln \frac{\beta}{2p}-1+\Big(1-\frac{\beta}{2}\Big)\frac{2}{\beta n}\Big)+\frac{1}{2}\ln\Big(\frac{\pi \beta n^2}{p}\Big)+o(1).
\end{align}
Therefore, 
\begin{align}\label{eq:p-term-2}
\frac{d_n}{p}\ln\Big(\frac{\beta n v_p}{2}\Big)
-\ln \Gamma\Big(1+\frac{d_n}{p}\Big)
=&-\frac{d_n}{p}\Big(\ln n -\ln(p v_p)-1+\Big(1-\frac{\beta}{2}\Big)\frac{2}{\beta n}\Big)+o(n)\notag\\
=&-\frac{\beta}{2p}n^2\ln n+\frac{\beta}{2p}\big(\ln(pv_p)+1\big)n^2\notag\\
&-\frac{1}{p}\Big(1-\frac{\beta}{2}\Big)n\ln n+\frac{\ln(pv_p)}{p}\Big(1-\frac{\beta}{2}\Big)n +o(n).
\end{align}

Combining Theorem~\ref{thm:LS}, Lemma~\ref{lem:cn} and \eqref{eq:p-term-2} with \eqref{eq:vol-asymp-parts} and collecting terms completes the proof of \eqref{eq:main-1} in Theorem~\ref{thm:main}.

For the proof of \eqref{eq:main-2} in Theorem~\ref{thm:main} we proceed analogously, replacing Theorem~\ref{thm:LS} by Theorem~\ref{thm:CKM}.
\end{proof}

\subsection*{Acknowledgement}

The author is grateful to Thomas Lebl\'e for a helpful discussion regarding the case $p=1$. Funded by the Deutsche Forschungsgemeinschaft (DFG, German Research Foundation) through the SPP 2265 Random Geometric Systems and under Germany's Excellence Strategy EXC 2044/2 - 390685587, Mathematics Münster: Dynamics-Geometry-Structure. This research was funded in whole or in part by the Austrian Science Fund (FWF) [Grant DOI: 10.55776/J4777]. For open access purposes, the author has applied a CC BY public copyright license to any author-accepted manuscript version arising from this submission.

\bibliographystyle{abbrv}
\bibliography{schatten}

\end{document}